\documentclass[12pt,leqno]{article}
\usepackage{amsfonts,amsthm,amsmath}
\newtheorem{thm}{Theorem}[section]
\newtheorem{prop}[thm]{Proposition}
\newtheorem{lem}[thm]{Lemma}
\newtheorem{cor}[thm]{Corollary}
\theoremstyle{remark}
\newtheorem{rem}[thm]{Remark}

\newcommand{\ZZ}{\mathbb{Z}}
\newcommand{\RR}{\mathbb{R}}

\newcommand{\0}{\mathbf{0}}

\begin{document}
\title{On the Existence of Extremal Type~II
$\ZZ_{2k}$-Codes\footnote{This work was supported by JST PRESTO program.}
}


\author{
Masaaki Harada\thanks{
Department of Mathematical Sciences,
Yamagata University,
Yamagata 990--8560, Japan. 
email: mharada@sci.kj.yamagata-u.ac.jp}
and 
Tsuyoshi Miezaki\thanks{Department of Mathematics, 
Oita National College of Technology, 
1666 Oaza-Maki, Oita, 870--0152, Japan. email: miezaki@oita-ct.ac.jp
}
}

\date{}

\maketitle

\begin{center}
{\sl In memory of Boris Venkov}
\end{center}

\begin{abstract}
For lengths $8,16$ and $24$, it is known that
there is an extremal Type~II $\ZZ_{2k}$-code
for every positive integer $k$.
In this paper, 
we show that there is an extremal Type~II $\ZZ_{2k}$-code of 
lengths $32,40,48,56$ and $64$ for every positive integer $k$.
For length $72$,  it is also shown that
there is an extremal Type~II $\ZZ_{4k}$-code 
for every positive integer $k$ with $k \ge 2$.
\end{abstract}

{\small
\noindent
{\bfseries Key Words:}
self-dual code, unimodular lattice, $k$-frame, theta series 

\noindent
2010 {\it Mathematics Subject Classification}. 
Primary 94B05; Secondary 11H71; Tertiary 11F11.\\ \quad
}

\section{Introduction}\label{Sec:1}

As described in~\cite{RS-Handbook},
self-dual codes are an important class of linear codes for both
theoretical and practical reasons.
It is a fundamental problem to classify self-dual codes
of modest lengths 
and determine the largest minimum weight among self-dual codes
of that length.
For binary doubly even self-dual codes,
much work has been done concerning this fundamental problem
(see~\cite{SPLAG, Huffman05, RS-Handbook}).
Binary doubly even self-dual codes are often called Type~II codes.
For general $k$,
Type~II $\ZZ_{2k}$-codes were defined in~\cite{BDHO} 
as self-dual codes with the property that all
Euclidean weights are divisible by $4k$,
where $\ZZ_{m}$ is the ring 
of integers modulo $m$.
By Construction A, 
Type~II $\ZZ_{2k}$-codes give even unimodular lattices.
It follows that a Type~II $\ZZ_{2k}$-code of length $n$ exists 
if and only if $n$ is divisible by eight (see~\cite{BDHO}).

Let $C$ be a Type~II $\ZZ_{2k}$-code of length $n \equiv 0\pmod 8$.
If $n \le 136$ 
then we have the following bound on the minimum Euclidean weight
$d_E(C)$ of $C$:
\begin{equation}\label{Eq:B}
d_E(C) \le 4k \left\lfloor \frac{n}{24} \right\rfloor +4k,
\end{equation}
for every positive integer $k$~\cite{BDHO, Z4-BSBM, HM,MS73}.
We say that a Type~II $\ZZ_{2k}$-code meeting the bound (\ref{Eq:B})
with equality is {\em extremal} for length $n \le 136$.

The existence of an extremal Type~II $\ZZ_{2k}$-code of 
length $n \le 64$ ($n \equiv 0 \pmod 8$)
is known for $k=1,2,\ldots,6$ (see~\cite[Table 1]{HM}).
The existence of a binary extremal Type~II code of length $72$
is a long-standing open question (see~\cite{Huffman05, RS-Handbook}).
Moreover, 
the existence of an extremal Type~II $\ZZ_{2k}$-code 
of length $72$ is not known for any positive integer $k$.
It is well known that
the binary extended Golay code is a binary
extremal Type~II code of length $24$.
For every positive integer $k \ge 2$,
it was shown in~\cite{Chapman, GH01} that
the Leech lattice contains a $2k$-frame\footnote{Recently, 
the second author~\cite{Miezaki}
has shown that the odd Leech lattice contains a $k$-frame for 
every positive integer $k$ with $k \ge 3$ by the approach
which is similar to that in~\cite{Chapman, GH01}.}.
Hence, there is an extremal Type~II $\ZZ_{2k}$-code of
length $24$  for every positive integer $k$. 
This motivates our investigation of the existence of
an extremal Type~II $\ZZ_{2k}$-code of larger lengths.
The main aim of this paper is to establish the existence of
an extremal Type~II $\ZZ_{2k}$-code of lengths $32,40,48,56$ and $64$
for every positive integer $k$.
This result yields that if $L$ is an extremal even unimodular lattice 
with theta series $\theta_L(q) = \sum_{m=0}^{\infty}A_{2m}q^{2m}$
for  dimensions $n=32,40,48,56$ and $64$
then $A_{2m} > 0$ for $m \ge \left\lfloor \frac{n}{24} 
\right\rfloor+1$.
For length $72$,  it is also shown that
there is an extremal Type~II $\ZZ_{4k}$-code 
for every positive integer $k$ with $k \ge 2$.
It is known that a unimodular lattice $L$
contains a $k$-frame if and only if there exists 
a self-dual $\ZZ_{k}$-code $C$ 
such that $L$ is isomorphic to 
the lattice obtained from $C$ by Construction A.
The powerful tool in the study of this paper
is to consider the existence
of $2k$-frames in some extremal even unimodular lattices.

This paper is organized as follows. In Section~\ref{sec:Pre}, we
give definitions and some basic properties of self-dual codes,
weighing matrices and unimodular lattices used in this paper.
In Section~\ref{sec:frame}, we provide methods for
constructing $m$-frames in unimodular lattices, which
are constructed from some self-dual $\ZZ_k$-codes by Construction A.
Using the theory of modular forms (see~\cite{Miyake} for details),
we also derive number theoretical results to give
infinite families of $m$-frames based on the above methods
(Theorems~\ref{thm:prime23} and~\ref{thm:prime29}).
In Section~\ref{sec:3240}, by the approach,
which is similar to that used in~\cite{Chapman, GH01}, 
we show the existence of
an extremal Type~II $\ZZ_{2k}$-code of lengths $32$ and $40$
for every positive integer $k$
(Theorem~\ref{thm:3240}).
This is done by finding $2k$-frames in some extremal even unimodular
lattices constructed from extremal Type~II $\ZZ_4$-codes,
along with constructing extremal Type~II $\ZZ_{22}$-codes.
In Sections~\ref{sec:48}, \ref{sec:56} and \ref{sec:64},
for every positive integer $k$, we show the existence of
an extremal Type~II $\ZZ_{2k}$-code of lengths $48,56$ and $64$, 
respectively (Theorems~\ref{thm:48}, \ref{thm:56} and \ref{thm:64}).
Our approach in Sections~\ref{sec:48}, \ref{sec:56} and \ref{sec:64}
is similar to that in Section~\ref{sec:3240}, but
we have to find self-dual $\ZZ_k$-codes 
$(k=3,5)$, instead of extremal Type~II $\ZZ_4$-codes,
since extremal even unimodular lattices have minimum norm $6$.
Hence, 
some number theoretical results established in Section~\ref{sec:frame}
are required for these lengths.
The first example of an extremal Type~II $\ZZ_{2k}$-code 
of length $n$ is also explicitly found for 
$(n,2k)=(48,14),(48,46),(56,14),(56,34), (56,46), (64,14)$ and $(64,46)$. 
Some of examples are used to complete the proofs of
Theorems~\ref{thm:48} and \ref{thm:56}.
In Section~\ref{sec:72}, it is shown that
there is an extremal Type~II $\ZZ_{4k}$-code of length
$72$ for every positive integer $k$ with $k \ge 2$
(Theorem~\ref{thm:72}),
by finding a $4k$-frame in 
the extremal even unimodular lattice in dimension $72$, 
which has been recently found by Nebe~\cite{Nebe72}.
Finally, in Section~\ref{sec:T}, we discuss the
positivity of coefficients of theta series of extremal even unimodular
lattices.  

Most of computer calculations in this paper were
done by {\sc Magma}~\cite{Magma}.

\section{Preliminaries}\label{sec:Pre}

In this section, we give definitions and some basic properties 
of self-dual codes, weighing matrices and unimodular lattices 
used in this paper.
\subsection{Self-dual codes}
Let $\ZZ_{k}$ be the ring 
of integers modulo $k$, where $k$ 
is a positive integer. 
In this paper, we always assume that $k\geq 2$ and 
we take the set $\ZZ_{k}$ to be 
$\{0,1,\ldots,k-1\}$.
A $\ZZ_{k}$-code $C$ of length $n$
(or a code $C$ of length $n$ over $\ZZ_{k}$)
is a $\ZZ_{k}$-submodule of $\ZZ_{k}^n$.
A $\ZZ_2$-code and a $\ZZ_3$-code
are called binary and ternary, respectively.
A $\ZZ_{k}$-code $C$ is {\em self-dual} if $C=C^\perp$, where
the dual code $C^\perp$ of $C$ is defined as 
$C^\perp = \{ x \in \ZZ_{k}^n \mid x \cdot y = 0$ for all $y \in C\}$
under the standard inner product $x \cdot y$. 
The Euclidean weight of a codeword $x=(x_1,\ldots,x_n)$ of $C$ is
$\sum_{\alpha=1}^{\lfloor k/2 \rfloor}n_\alpha(x) \alpha^2$, 
where $n_{\alpha}(x)$ denotes
the number of components $i$ with $x_i \equiv \pm \alpha \pmod k$ 
$(\alpha=1,2,\ldots,\lfloor k/2 \rfloor)$.
The minimum Euclidean weight $d_E(C)$ of $C$ is the smallest Euclidean
weight among all nonzero codewords of $C$.

Binary doubly even self-dual codes 
have been widely studied
(see~\cite{SPLAG,Huffman05, RS-Handbook}).
Binary doubly even self-dual codes are often called Type~II codes.
For general $k$,
{\em Type~II} $\ZZ_{2k}$-codes were defined in~\cite{BDHO} 
as self-dual codes with the property that all
Euclidean weights are divisible by $4k$
(see~\cite{Z4-BSBM} and~\cite{Z4-HSG} for $\ZZ_4$-codes).
Type~II $\ZZ_{2k}$-codes are an important class of
self-dual $\ZZ_{2k}$-codes, since 
Type~II $\ZZ_{2k}$-codes give even unimodular lattices 
by Construction A.
It is known that a Type~II $\ZZ_{2k}$-code of length $n$ exists 
if and only if $n$ is divisible by eight~\cite{BDHO}.

Let $C$ be a Type~II $\ZZ_{2k}$-code of length $n \equiv 0 \pmod 8$.
The bound (\ref{Eq:B}) is established
for $k=1$~\cite{MS73},
for $k=2$~\cite{Z4-BSBM},
for $k=3,4,5,6$~\cite{HM}. 
For $k \ge 3$, the bound (\ref{Eq:B}) is known under the
assumption that $\lfloor n/24 \rfloor\leq k-2$~\cite{BDHO}.
Therefore, if $n \le 136$ then we have  (\ref{Eq:B}) for 
every positive integer $k$.
We say that a Type~II $\ZZ_{2k}$-code of length $n \le 136$
meeting the bound (\ref{Eq:B})
with equality is {\em extremal}.
The existence of an extremal Type~II $\ZZ_{2k}$-code of 
length $n \le 64$ ($n \equiv 0 \pmod 8$)
is known for $k=1,2,\ldots,6$ (see~\cite[Table 1]{HM}).
For lengths $8,16$ and $24$, it is known that
there is an extremal Type~II $\ZZ_{2k}$-code
for every positive integer $k$.

\subsection{Weighing matrices and negacirculant matrices}
A {\em weighing matrix} $W$ of order $n$ and weight $k$ is an
$n \times n$ $(1,-1,0)$-matrix $W$ such that
$W W^T=kI$, where $I$ is
the identity matrix and $W^T$ denotes the transpose of $W$
(see~\cite{GS} for details of weighing matrices).
A weighing matrix $W$ is called {\em skew-symmetric} if $W^T=-W$.

In this paper, 
to construct self-dual $\ZZ_k$-codes, we use
weighing matrices $M$ of order $n$ and 
weight $\equiv -1 \pmod k$ and 
$n \times n$ 
$(0,\pm 1,\pm2,\ldots,\pm \lfloor k/2  \rfloor)$-matrices $M$
with $M M^T=mI$ and $m \equiv -1 \pmod k$ as follows.
%
Let $C_k(M)$ be  the $\ZZ_k$-code of length $2n$
with generator matrix
$\left(\begin{array}{cc}
I & M
\end{array}\right)$, where 
the entries of the matrix are regarded as elements of $\ZZ_k$.
Then it is easy to see that $C_k(M)$ is self-dual.

An  $n \times n$ matrix $M$ is circulant and 
{\em negacirculant} if
$M$ has the following form:
\[
\left( \begin{array}{ccccc}
r_0     &r_1     & \cdots &r_{n-1} \\
r_{n-1}&r_0     & \cdots &r_{n-2} \\
r_{n-2}&r_{n-1}& \cdots &r_{n-3} \\
\vdots  & \vdots && \vdots\\
r_1    &r_2    & \cdots&r_0
\end{array}
\right) \text{ and }
\left( \begin{array}{ccccc}
r_0     &r_1     & \cdots &r_{n-1} \\
-r_{n-1}&r_0     & \cdots &r_{n-2} \\
-r_{n-2}&-r_{n-1}& \cdots &r_{n-3} \\
\vdots  & \vdots && \vdots\\
-r_1    &-r_2    & \cdots&r_0
\end{array}
\right),
\]
respectively.
Most of matrices including weighing matrices constructed in this paper 
are based on negacirculant matrices.
For example, 
in order to construct self-dual $\ZZ_k$-codes of length $4n$,
we often consider a generator 
matrix  of the form:
\begin{equation} \label{eq:GM}
\left(
\begin{array}{ccc@{}c}
\quad & {\Large I} & \quad &
\begin{array}{cc}
A & B \\
-B^T & A^T
\end{array}
\end{array}
\right),
\end{equation}
where $A$ and $B$ are $n \times n$ negacirculant matrices.
It is easy to see that the code is self-dual if
$AA^T+BB^T=-I$.

\subsection{Unimodular lattices}\label{sec:2U}
A (Euclidean) lattice $L \subset \RR^n$
in dimension $n$
is {\em unimodular} if
$L = L^{*}$, where
the dual lattice $L^{*}$ of $L$ is defined as
$\{ x \in {\RR}^n \mid (x,y) \in \ZZ \text{ for all }
y \in L\}$ under the standard inner product $(x,y)$.
Two lattices $L$ and $L'$ are {\em isomorphic}, denoted $L \cong L'$,
if there exists an orthogonal matrix $A$ with
$L' = L \cdot A$.
The norm of a vector $x$ is defined as $(x, x)$.
The minimum norm $\min(L)$ of a unimodular
lattice $L$ is the smallest norm among all nonzero vectors of $L$.
A unimodular lattice with even norms is said to be {\em even}, 
and that containing a vector of odd norm is said to be odd.
An odd  unimodular lattice exists for every dimension.
Indeed, $\ZZ^n$ is an odd  unimodular lattice in dimension $n$.

Let $L$ be an even unimodular lattice in dimension $n$.
Denote by $\theta_L(q) = \sum_{x \in L} q^{(x,x)}=
1+\sum_{m=1}^{\infty}A_{2m}q^{2m}$
the theta series of $L$,
that is, $A_{2m}$ is the number of vectors of norm $2m$ in $L$.
In this subsection, we simply consider
the theta series as formal power series.
Then there exist integers $a_0=1$, $a_1,\ldots, 
a_{\lfloor \frac{n}{24}\rfloor}$ so that 
\begin{equation}\label{eq:T2}
\theta_L(q)=\sum_{r=0}^{\lfloor \frac{n}{24}\rfloor}
a_rE_4(q)^{\frac{n}{8}-3r}\Delta(q)^r,
\end{equation}
where
$E_4(q)=1+240\sum_{m=1}^{\infty}\sigma_3(m)q^{2m}$, 
$\Delta(q)=q^2\prod_{m=1}^{\infty}(1-q^{2m})^{24}$, and 
$\sigma_3(m)=\sum_{0<d|m}d^3$~{\cite[p.~193]{SPLAG}}
(see also~\cite[Proposition~15]{Pache}). 
As a consequence of (\ref{eq:T2}), 
an even unimodular lattice in dimension $n$
exists if and only
if $n \equiv 0 \pmod 8$. 
It was shown in~\cite{Siegel}
that the coefficient $A_{2 \lfloor \frac{n}{24} \rfloor+2}$ 
is always positive when
$A_2=A_4=\cdots =A_{2\lfloor \frac{n}{24} \rfloor} =0$
(see also \cite{MOS75}).
Hence, 
the minimum norm $\min(L)$ of $L$ is bounded by
\[
\min(L) \le 2 \left\lfloor \frac{n}{24} \right\rfloor+2.
\]
We say that an even unimodular lattice meeting the upper
bound is {\em extremal}.
It follows that the theta series of an extremal even unimodular lattice in
each dimension is uniquely determined (see~{\cite[p.~193]{SPLAG}}). 

The existence of an extremal even  unimodular lattice
is known for dimension $n \le 80$ and $n \equiv 0 \pmod 8$
(see~\cite{S-http}).
The existence of an extremal even unimodular lattice
in dimension $72$ was a long-standing open question.
Recently, the first example of such a lattice has been
found by Nebe~\cite{Nebe72}.
On the other hand, Mallows, Odlyzko and Sloane~\cite{MOS75} showed 
that there is no extremal even unimodular lattice in 
dimension $n$ for all sufficiently large $n$
by verifying that $A_{2 \lfloor \frac{n}{24} \rfloor+4}$ is negative. 
Recently, 
it has been shown in~\cite{JR} that 
there is no extremal even unimodular lattice in dimension $n$ 
with $n>163264$ by verifying that
the largest $n$ for which all $A_i$ are non-negative is $163264$.


Two lattices $L$ and $L'$ are {\em neighbors} if
both lattices contain a sublattice of index $2$
in common.
Let $L$ be an odd unimodular lattice and let $L_0$ denote its
sublattice of vectors of even norms.
Then $L_0$ is a sublattice of $L$ of index $2$.
$L_0^* \setminus L$ is called the shadow $S$ of $L$~\cite{CS-odd}.
There are cosets $L_1,L_2,L_3$ of $L_0$ such that
$L_0^* = L_0 \cup L_1 \cup L_2 \cup L_3$, where
$L = L_0  \cup L_2$ and $S = L_1 \cup L_3$.
If $L$ is an odd unimodular lattice in dimension divisible by eight,
then $L$ has two even unimodular neighbors
of $L$, namely, $L_0 \cup L_1$ and $L_0 \cup L_3$.

\subsection{Construction A and $k$-frames}

We give a method to construct 
unimodular lattices from self-dual $\ZZ_{k}$-codes, which 
is referred to as {\em Construction A} (see~\cite{{BDHO},{HMV}}). 
Let $\rho$ be a map from $\ZZ_{k}$ to $\ZZ$ sending $0, 1, \ldots , k-1$ 
to $0, 1, \ldots , k-1$, respectively.
If $C$ is a  self-dual $\ZZ_{k}$-code of length $n$, then 
the lattice 
\[
A_{k}(C)=\frac{1}{\sqrt{k}}\{\rho (C) +k \ZZ^{n}\} 
\]
is a unimodular lattice in dimension $n$, where 
$
\rho (C)=\{(\rho (c_{1}), \ldots , \rho (c_{n})) 
\mid (c_{1}, \ldots , c_{n}) \in C\}. 
$
The minimum norm of $A_{k}(C)$ is $\min\{k, d_{E}(C)/k\}$.
Moreover, $C$ is a Type~II $\ZZ_{2k}$-code if and only if
$A_{2k}(C)$ is an even unimodular lattice~\cite{BDHO}.

A set $\{f_1, \ldots, f_{n}\}$ of $n$ vectors $f_1, \ldots, f_{n}$ of 
a unimodular lattice $L$ in dimension $n$ with
$ ( f_i, f_j ) = k \delta_{i,j}$
is called a {\em $k$-frame} of $L$,
where $\delta_{i,j}$ is the Kronecker delta.
It is known that a unimodular lattice $L$ contains a $k$-frame 
if and only if there exists a self-dual $\ZZ_{k}$-code $C$ with 
$A_{k}(C) \cong L$.
In addition, an even unimodular lattice $L$ contains a $2k$-frame 
if and only if there exists a Type~II $\ZZ_{2k}$-code $C$ 
with $A_{2k}(C) \cong L$ (see~\cite{{Chapman},{HMV}}). 
Hence, we directly have the following lemma.

\begin{lem}\label{lem:2-1}
Suppose that $n \le 136$ and $k \ge  \lfloor n/24 \rfloor+1$.
There is an extremal even unimodular lattice in 
dimension $n$ containing a $2k$-frame
if and only if there is an extremal Type~II $\ZZ_{2k}$-code
of length $n$.
\end{lem}

\section{Methods for constructing $k$-frames}\label{sec:frame}


By the following lemma, it is enough to consider 
a $p$-frame (resp.\ $2p$-frame)
of an odd (resp.\ even) unimodular lattice
for each prime $p$.

\begin{lem}[{Chapman~\cite[Lemma 5.1]{Chapman}}]
\label{lem:frame}
If a lattice $L$ in dimension $n\equiv 0 \pmod 4$ contains a $k$-frame, then
$L$ contains a $km$-frame for every positive integer $m$.
\end{lem}

As a consequence, we have the following:

\begin{lem}\label{lem:neighbor}
If an odd unimodular lattice $L$ in dimension $n\equiv 0 \pmod 8$
contains a $k$-frame,
then both even unimodular neighbors
$L_0 \cup L_1$ and $L_0 \cup L_3$ contain a $2k$-frame.
\end{lem}

Chapman~\cite{Chapman} showed that the Leech lattice $\Lambda_{24}$
contains a $2k$-frame for every positive integer $k$
with $k \ge 2$ and $k\ne 11^{\ell}$, 
where $\ell$ is a positive integer.
This was established from a construction of 
$2k$-frames of $\Lambda_{24}$, which is the case 
$(n,m)=(12,11)$ of 
Proposition~\ref{prop:constZ4},
using some extremal
Type~II $\ZZ_4$-code $C$ with $A_4(C) \cong \Lambda_{24}$,
along with Lemma~\ref{lem:frame} and 
Theorem~\ref{thm:prime11}.

\begin{prop}\label{prop:constZ4}
Let $W$ be a skew-symmetric weighing matrix of order 
$n \equiv 0 \pmod 4$ and weight $m \equiv 3 \pmod 8$.
Let $\tilde{C}_4(W)$ be the Type~II $\ZZ_4$-code of length $2n$
with generator matrix
$\left(\begin{array}{cc}
I & W+2I
\end{array}\right)$, where 
the entries of the matrix are regarded as elements of $\ZZ_4$.
Let $a,b,c$ and $d$ be integers with 
$c \equiv 2a+b \pmod 4$ and
$d \equiv a+2b \pmod 4$. 
Then the set of $2n$ rows of the following matrix
\[
\tilde{F}(W)=
\frac{1}{2}
\left(
\begin{array}{cc}
aI+bW & cI+dW \\
-cI+dW & aI-bW
\end{array}
\right)
\]
forms a $\frac{1}{4}(a^2+m b^2+c^2+m d^2)$-frame of
the even unimodular lattice $A_4(\tilde{C}_4(W))$.
\end{prop}
\begin{proof}
By~\cite[Proposition 4]{Z4-H}, $\tilde{C}_4(W)$ is Type~II.
Then $A_4(\tilde{C}_4(W))$ is an even unimodular lattice.
Since $\tilde{C}_4(W)$ is self-dual and $W$ is skew-symmetric,
the matrix 
$\left(\begin{array}{cc}
W+2I & I
\end{array}\right)$, which is a parity-check matrix,
is also a generator matrix of $\tilde{C}_4(W)$.
Hence, it follows from 
$c \equiv 2a+b \pmod 4$ and $d \equiv a+2b \pmod 4$ that
the rows of $\tilde{F}(W)$ are vectors of
 $A_4(\tilde{C}_4(W))$.
Since $\tilde{F}(W) \tilde{F}(W)^T=\frac{1}{4}(a^2+mb^2+c^2+m d^2)I$,
the set of the $2n$ rows of $\tilde{F}(W)$ forms a
$\frac{1}{4}(a^2+m b^2+c^2+m d^2)$-frame of $A_4(\tilde{C}_4(W))$.
\end{proof}

\begin{rem}
It follows from the assumption 
that $a^2+m b^2+c^2+m d^2 \equiv 0 \pmod8$.
\end{rem}


\begin{thm}[{Chapman~\cite[Theorem~5.2]{Chapman}}]
\label{thm:prime11}
There are integers $a,b,c$ and $d$ satisfying
$c \equiv 2a+b \pmod 4$,
$d \equiv a+2b \pmod 4$ and 
$2p=\frac{1}{4}(a^2+11b^2+c^2+11d^2)$ for each odd prime $p \ne 11$.
\end{thm}


For dimensions $2n=32$ and $40$,
we are able to employ the approach which is similar to that 
in~\cite{Chapman, GH01},
by considering skew-symmetric weighing matrices of order $n$ and weight $11$
(Section~\ref{sec:3240}).


On the other hand,
for dimensions $48,56,$ and $64$, 
no $\ZZ_4$-code gives an extremal even unimodular lattice 
by Construction A.
Hence, we consider extremal even unimodular lattices
which are even unimodular neighbors of some odd unimodular lattices 
constructed from self-dual $\ZZ_k$-codes $(k=3,5)$
(Sections~\ref{sec:48}, \ref{sec:56} and \ref{sec:64}).
To do this, we provide the following modification of
Proposition~\ref{prop:constZ4}.

\begin{prop}\label{prop:constZk}
Let $k$ be a positive integer with $k \ge 2$.
Let $M$ be an $n \times n$ 
$(0,\pm 1,\pm2,\ldots,\pm \lfloor k/2 \rfloor)$-matrix
satisfying $M^T=-M$ and $M M^T= mI$, where $m \equiv -1 \pmod k$.
Let $C_k(M)$ be the self-dual $\ZZ_k$-code of length $2n$
with generator matrix
$\left(\begin{array}{cc}
I & M
\end{array}\right)$, where 
the entries of the matrix are regarded as elements of $\ZZ_k$.
Let $a,b,c$ and $d$ be integers with 
$a \equiv d \pmod k$ and 
$b \equiv c \pmod k$.
Then the set of $2n$ rows of the following matrix 
\[
F(M)=
\frac{1}{\sqrt{k}}
\left(
\begin{array}{cc}
aI+bM & cI+dM \\
-cI+dM & aI-bM
\end{array}
\right)
\]
forms a $\frac{1}{k}(a^2+m b^2+c^2+m d^2)$-frame of 
the unimodular lattice $A_k(C_k(M))$.
\end{prop}
\begin{proof}
Since 
$M M^T= mI$ with $m \equiv -1 \pmod k$,
$C_k(M)$ is a self-dual $\ZZ_k$-code of length $2n$.
Thus, $A_k(C_k(M))$ is a unimodular lattice.
Since $C_k(M)$ is self-dual and $M^T=-M$,
the matrix 
$\left(\begin{array}{cc}
M & I
\end{array}\right)$ is also a generator matrix of $C_k(M)$.
Hence, it follows from 
$a \equiv d \pmod k$ and $b \equiv c \pmod k$ that
all rows of the matrix $F(M)$
are vectors of $A_k(C_k(M))$.
Since $F(M) F(M)^T=\frac{1}{k}(a^2+mb^2+c^2+md^2)I$,
the set of the $2n$ rows of $F(M)$ forms a
$\frac{1}{k}(a^2+m b^2+c^2+m d^2)$-frame of $A_k(C_k(M))$.
\end{proof}

\begin{rem}
It follows from the assumption 
that $a^2+m b^2+c^2+m d^2 \equiv 0 \pmod k$.
\end{rem}

For dimensions $48,56$ and $64$, we consider cases
$(m,k)=(23,3),(29,5)$ in Proposition~\ref{prop:constZk}.
Hence, we need the following modifications of
Theorem~\ref{thm:prime11}.
The proofs of Theorems~\ref{thm:prime23} and \ref{thm:prime29}
need some facts of modular forms for congruence subgroups. 
Our notation and terminology for modular forms
follow from~\cite{Miyake}
(see~\cite{Miyake} for undefined terms).

\begin{thm}\label{thm:prime23}
There are integers $a,b,c$ and $d$ satisfying
$a \equiv d \pmod 3$,
$b \equiv c \pmod 3$ and 
$p=\frac{1}{3}(a^2+23b^2+c^2+23d^2)$ for each prime $p \ne 2,5,7, 23$.
\end{thm}
\begin{proof}
Consider the following lattice in dimension $4$:
\[
L_1=\{(a, b, c, d)\in\ZZ^4
\mid a\equiv d\pmod{3} \text{ and } b\equiv c\pmod{3}\}.
\]
Here, we consider the inner product $\langle x,y \rangle_1$
induced by 
$(a^2+23b^2+c^2+23d^2)/3$, instead of the standard inner product.
This lattice is spanned by 
$(3,0,0,0)$, $(0,3,0,0)$, $(1,0,0,1)$, and $(0,1,1,0)$
with Gram matrix:
\[M_1=
\begin{pmatrix}
3&  0& 1&  0\\
0& 69& 0& 23\\
1&  0& 8&  0\\
0& 23& 0&  8
\end{pmatrix}.
\]
We have verified by {\sc Magma} that 
the lattice $L_1$ has the following theta series: 
\begin{align*}
\theta_{L_1}(q)&=\sum_{x\in L_1}q^{\langle x,x \rangle_1}
=1 + 4q^3 + 4q^6 + 4q^8 + 4q^9 + 8q^{11} + \cdots\\ 
&=\sum_{n=0}^{\infty}a_1(n)q^n\ (\text{say}). \nonumber
\end{align*}
Since $23M_1^{-1}$ has integer entries and 
$\det M_1=23^2$, 
$\theta_{L_1}(z)$ is a modular form (of weight $2$) for $\Gamma_0(92)$~\cite[Corollary~4.9.2]{Miyake}, 
where $q=e^{2\pi i z}$, $z$ is in the upper half plane, and 
\[
\Gamma_0(N)=
\left\{
\begin{pmatrix}
a&b\\
c&d
\end{pmatrix}
\in SL_2(\ZZ) \mid   c\equiv 0\pmod{N}
\right\}.
\]

It is known that the dimension of the space of cusp forms of weight $2$ 
for $\Gamma_0(23)$ is two and using {\sc Magma}
we have found some basis as follows: 
\begin{align*}
f(z)&=q - q^3 - q^4 - 2q^6 + 2q^7 - q^8 + 2q^9 + 2q^{10} - 4q^{11} +\cdots\\
&=\sum_{n=1}^{\infty}c_f(n)q^n\ (\text{say}),\\
g(z)&=q^2 - 2q^3 - q^4 + 2q^5 + q^6 + 2q^7 - 2q^8 - 2q^{10} - 2q^{11}+\cdots\\
&=\sum_{n=1}^{\infty}c_g(n)q^n\ (\text{say}). 
\end{align*}
Let 
$\eta(z)=q^{\frac{1}{24}}\prod_{n=1}^{\infty}(1-q^n)$
be the Dedekind $\eta$-function. 
Then 
\[
\frac{\eta(4z)^8}{\eta(2z)^4}=\sum_{n=1}^{\infty}\sigma_1(2n-1)q^{2n-1} 
\]
is a modular form for $\Gamma_0(4)$, where 
$\sigma_1(n)=\sum_{p|n}p$ (see~\cite[p.~145, Problem~10]{Kob}). 
We define a modular form $h_{92}(z)$ for $\Gamma_0(92)$ as follows: 
\begin{multline*}
h_{92}(z)=\frac{4}{11}\Bigg{(}
\frac{\eta(4z)^8}{\eta(2z)^4}
-23\frac{\eta(92z)^8}{\eta(46z)^4}
-f(z)
+2f(2z)
-4f(4z)\\
-3g(z)
-5g(2z)
-12g(4z)\Bigg{)}
=\sum_{n=0}^{\infty}b_1(n)q^n\ (\text{say}).
\end{multline*}
Let 
\[
\chi_2(n)=
\left\{
\begin{array}{cl}
0 & \text{ if }n\equiv 0\pmod{2}\\
1 & \text{ otherwise. }
\end{array}
\right.
\]
Then 
\begin{align*}
(\theta_{L_1}(z))_{\chi_2}=\sum_{n=0}^{\infty}\chi_2(n)a_1(n)q^n
&=4q^3 + 4q^9 + 8q^{11} + \cdots \text{ and }\\
(h_{92}(z))_{\chi_2}=\sum_{n=0}^{\infty}\chi_2(n)b_1(n)q^n
&=4q^3 + 4q^9 + 8q^{11} + \cdots 
\end{align*}
are modular forms with character $\chi_2$
for $\Gamma_0(368)$~{\cite[p.~127, Proposition~17]{Kob}}. 
Using Theorem~7 in~\cite{Murty} and the fact that 
the genus of $\Gamma_0(368)$ is $43$,  
the verification by {\sc Magma} that $\chi_2(n)a_1(n)=\chi_2(n)b_1(n)$ for $n\leq 86$
shows 
\[
(\theta_{L_1}(z))_{\chi_2}=(h_{92}(z))_{\chi_2}. 
\]
Hence, for each odd prime $p$ with $p\ne 23$, we have
\begin{align}
a_1(p)=\frac{4}{11}(p+1-(c_f(p)+3c_g(p))). \label{eqn:coef}
\end{align}

Set $h_{1}(z)$ and $h_{2}(z)$ as follows: 
\begin{align*}
h_1(z)&=f(z)+\frac{-1+\sqrt{5}}{2}g(z)=\sum_{n=1}^{\infty}c_{h_1}(n)q^n\
 (\text{say}) \text{ and} \\
h_2(z)&=f(z)+\frac{-1-\sqrt{5}}{2}g(z)=\sum_{n=1}^{\infty}c_{h_2}(n)q^n\ (\text{say}). 
\end{align*}
Let $T(n)$ be the Hecke operator considered on the space 
of modular forms for $\Gamma_0(23)$ 
(see {\cite[p.~161, Proposition~37]{Kob}}). 
Then, by~{\cite[Proposition 9.15]{Knapp}}
we have 
\[
T(2)f(z)=g(z) \text{ and } T(2)g(z)=f(z)-g(z). 
\]
Namely, $h_1(z)$ and $h_2(z)$ are eigen forms for 
$T(2)$.  Since the algebra of Hecke operators is 
commutative~\cite[Theorem~4.5.3]{Miyake}, 
$h_1(z)$ and $h_2(z)$ are normalized Hecke eigen forms. 
In addition, for each prime $p$ and $i = 1,2$, 
\[
|c_{h_i}(p)|\leq 2\sqrt{p} 
\]
(see~\cite[p.~164]{Kob}).
Since
$
f(z)=
\frac{5+\sqrt{5}}{10} h_1(z) 
+\frac{5-\sqrt{5}}{10} h_2(z) 
$
and 
$
g(z)=
\frac{\sqrt{5}}{5} h_1(z) 
+\frac{-\sqrt{5}}{5} h_2(z),
$
we have
\[
f(z)+3g(z)=
\frac{5+7\sqrt{5}}{10} h_1(z) 
+\frac{5-7\sqrt{5}}{10} h_2(z).
\]
Hence, for each prime $p$, we have 
\[
|c_f(p)+3c_g(p)|\leq 
\left(\frac{5+7\sqrt{5}}{10} +\frac{-5+7\sqrt{5}}{10} \right)
2\sqrt{p}.
\]
Using (\ref{eqn:coef}), $a_1(p)$ is bounded below by 
\begin{align}
\frac{4}{11}\left(p+1-
\frac{14\sqrt{p}}{\sqrt{5}} \right).
\label{eqn:bound}
\end{align}
Hence, (\ref{eqn:bound}) is positive for $p > 37$, 
namely, $a_1(p)>0$ for $p > 37$. 
We have verified by {\sc Magma} that
$a_1(p)>0$ for each prime $p$ with $p \le 37$ and 
$p \ne 2,5,7, 23$,
where $a_1(p)$ is listed in Table~\ref{Tab:an} for
a prime $p \le 37$.
\end{proof}

\begin{table}[thb]
\caption{Coefficients $a_1(p)$ for primes $p \le 37$}
\label{Tab:an}
\begin{center}
{\small
\begin{tabular}{cc|cc|cc|cc}
\noalign{\hrule height0.8pt}
$p$ & $a_1(p)$ & $p$ & $a_1(p)$ & 
$p$ & $a_1(p)$ & $p$ & $a_1(p)$ \\
\hline
$ 2$ & $ 0$&$ 7$ & $ 0$&$17 $& $ 8$&$29 $& $12$\\
$ 3$ & $ 4$&$11 $& $ 8$&$19 $& $ 8$&$31 $& $ 4$\\
$ 5$ & $ 0$&$13 $& $ 4$&$23 $& $ 0$&$37 $& $16$\\
\noalign{\hrule height0.8pt}
\end{tabular}
}
\end{center}
\end{table}

\begin{thm}\label{thm:prime29}  
There are integers $a,b,c$ and $d$ satisfying
$a \equiv d \pmod 5$,
$b \equiv c \pmod 5$ and 
$p=\frac{1}{5}(a^2+29b^2+c^2+29d^2)$ for each prime $p \ne 2, 3, 7, 17, 23$.
\end{thm}
\begin{proof}
We follow the same line as in the previous proof. 
Consider the following lattice in dimension $4$:
\[
L_{2}=\{(a, b, c, d)\in\ZZ^4
\mid a\equiv d\pmod{5} \text{ and } b\equiv c\pmod{5}\}.
\]
Here, we consider the inner product $\langle x,y \rangle_2$
induced by 
$(a^2+29b^2+c^2+29d^2)/5$.
This lattice is spanned by 
$(5,0,0,0)$, $(0,5,0,0)$, $(1,0,0,1)$, and $(0,1,1,0)$
with Gram matrix:
\[M_{2}=
\begin{pmatrix}
5&  0& 1&  0\\
0& 145& 0& 29\\
1&  0& 6&  0\\
0& 29& 0&  6
\end{pmatrix}.
\]
We have verified by {\sc Magma} that 
$L_{2}$  has the following theta series: 
\begin{align*}
\theta_{L_2}(q)&=\sum_{x\in L_{2}}q^{\langle x,x \rangle_2}
=1 + 4q^5 + 4q^6 + 4q^9 + 4q^{10} + 8q^{11} + \cdots \\
&=\sum_{n=0}^{\infty}a_{2}(n)q^n\ (\text{say}). \nonumber
\end{align*}
Since $29M_{2}^{-1}$ has integer entries and 
$\det M_{2}=29^2$, $\theta_{L_2}(z)$
is a modular form for $\Gamma_0(116)$~{\cite[p.~192]{Miyake}}.

It is known that the dimension of the space of cusp forms of weight $2$ 
for $\Gamma_0(116)$ is thirteen and using {\sc Magma} 
we have found some basis $f_1(z),\ldots, f_{13}(z)$ such that 
\[
f_i(z)=q^i+c_{i, 12}q^{12}+c_{i, 13}q^{13}+\cdots, 
\]
for $i=1,2,\ldots,11$, and
$f_i(z)=c_{i}q^{i}+\cdots$, 
where $c_{i} \neq 0$ for $i=12,13$. 
In particular, we use $f_i(z)\ (i=1,3,5,7,9,11,13)$, 
which are explicitly written as: 
\begin{align*}
f_1(z)&= q + 5q^{13} + 3q^{15} - 5q^{17} - 6q^{19} + 3q^{21} - 4q^{23} + 
3q^{25} - 6q^{27} -\cdots, \\
f_3(z)&=  q^3 + 10q^{13} + 7q^{15} - 10q^{17} - 17q^{19} + 10q^{21} - 4q^{23} 
+ 10q^{25} -   12q^{27} - \cdots,\\
f_5(z)&= q^5 + 8q^{13} + 5q^{15} - 7q^{17} - 14q^{19} + 7q^{21} - 4q^{23} 
+ 9q^{25} - 10q^{27} - \cdots,\\
f_7(z)&=  q^7 + 5q^{13} + 3q^{15} - 5q^{17} - 8q^{19} + 5q^{21} - 3q^{23} 
+ 5q^{25} - 6q^{27}  -\cdots,\\
f_9(z)&=  q^9 + 9q^{13} + 7q^{15} - 9q^{17} - 16q^{19} + 9q^{21} - 4q^{23} 
+ 10q^{25} -12q^{27} - \cdots,\\
f_{11}(z)&=  q^{11} + 6q^{13} + 5q^{15} - 7q^{17} - 9q^{19} + 5q^{21} - 4q^{23} + 6q^{25} - 7q^{27} - \cdots,\\
f_{13}(z)&=  14q^{13} + 11q^{15} - 13q^{17} - 24q^{19} + 15q^{21} - 6q^{23} 
+ 14q^{25} - 16q^{27} - \cdots. 
\end{align*}
For $i=1,3,5,7,9,11,13$, we denote by $c_{f_{i}}(n)$ 
the coefficient of $f_{i}(z)$ as follows: 
\[
f_{i}(z)=\sum_{n=1}^{\infty}c_{f_{i}}(n)q^n.
\]

We define a modular form $h_{116}(z)$ for $\Gamma_0(116)$ as follows: 
\begin{multline*}
h_{116}(z)=\frac{4}{15}\Bigg{(}
\frac{\eta(4z)^8}{\eta(2z)^4}
+29\frac{\eta(116z)^8}{\eta(58z)^4}
-f_1(z)
-4f_3(z)
+9f_5(z)\\
-8f_7(z)
+2f_9(z)
+18f_{11}(z)
-8f_{13}(z)
\Bigg{)}
=\sum_{n=0}^{\infty}b_{2}(n)q^n\ (\text{say}).
\end{multline*}
Then 
\begin{align*}
(\theta_{L_2}(z))_{\chi_2}&=\sum_{n=0}^{\infty}\chi_2(n)a_2(n)q^n
= 4q^5  + 4q^9  + 8q^{11} + \cdots 
\text{ and} \\
(h_{116}(z))_{\chi_2}&=\sum_{n=0}^{\infty}\chi_2(n)b_2(n)q^n
= 4q^5  + 4q^9  + 8q^{11} + \cdots 
\end{align*}
are modular forms with character $\chi_2$
for $\Gamma_0(464)$~{\cite[p.~127, Proposition~17]{Kob}}.
Using Theorem~7 in~\cite{Murty} and the fact that 
the genus of $\Gamma_0(464)$ is $55$, 
the verification by {\sc Magma} that 
$\chi_2(n)a_2(n)=\chi_2(n)b_2(n)$ for $n\leq 110$ 
shows
\[
(\theta_{L_2}(z))_{\chi_2}=(h_{116}(z))_{\chi_2}. 
\]
Hence, for each odd prime $p$ with $p\ne 29$, we have
\begin{multline}\label{eqn:coef56}
a_2(p)=\frac{4}{15}(p+1-(c_{f_1}(p)+4c_{f_3}(p)-9c_{f_5}(p) \\
+8c_{f_7}(p)-2c_{f_9}(p)-18c_{f_{11}}(p)+8c_{f_{13}}(p))). 
\end{multline}
Set $\hat{h}_i(z)\ (i=1,3,5,7,9,11,13)$ as follows: 
\begin{align}\label{mat:Hecke}
\begin{pmatrix}
\hat{h}_1(z)\\
\hat{h}_3(z)\\
\hat{h}_5(z)\\
\hat{h}_7(z)\\
\hat{h}_9(z)\\
\hat{h}_{11}(z)\\
\hat{h}_{13}(z)
\end{pmatrix}
&=
\begin{pmatrix}
 1 & -3 & 3 & 4 & 6 & -1 & -5 \\
 1 & 2 & -2 & 4 & 1 & -6 & 0 \\
 1 & 1 & 3 & -4 & -2 & 3 & -1 \\
 1 & -3 & -3 & -2 & 6 & -1 & 1 \\
 1 & -1 & 1 & -2 & -2 & -3 & 3 \\
 1 & 1+\sqrt{2} & -1 & -2 \sqrt{2} & 2 \sqrt{2} & 1-\sqrt{2} & -1-\sqrt{2} \\
 1 & 1-\sqrt{2} & -1 & 2 \sqrt{2} & -2 \sqrt{2} & 1+\sqrt{2} & -1+\sqrt{2}
\end{pmatrix}
\begin{pmatrix}
f_1(z)\\
f_3(z)\\
f_5(z)\\
f_7(z)\\
f_9(z)\\
f_{11}(z)\\
f_{13}(z)
\end{pmatrix}.
\end{align}
For $i=1,3,5,7,9,11,13$, we denote by $c_{\hat{h}_{i}}(n)$ 
the coefficient of $\hat{h}_{i}(z)$ as follows: 
\[
\hat{h}_{i}(z)=\sum_{n=1}^{\infty}c_{\hat{h}_{i}}(n)q^n.
\]
Let $T(n)$ be the Hecke operator considered on the space 
of modular forms for $\Gamma_0(116)$ 
(see {\cite[p.~161, Proposition~37]{Kob}}). 
Then, by~{\cite[Proposition 9.15]{Knapp}}
we have 
\begin{align*}
T(3)f_1(z)&=3f_3(z)+3f_5(z)+3f_7(z)-6f_9(z)-2f_{11}(z), \\
T(3)f_3(z)&=f_1(z)+7f_5(z)+10f_7(z)-9f_9(z)-6f_{11}(z)+f_{13}(z), \\
T(3)f_5(z)&=5f_5(z)+7f_7(z)-10f_9(z)-3f_{11}(z)+3f_{13}(z), \\
T(3)f_7(z)&=3f_5(z)+5f_7(z)-6f_9(z)-3f_{11}(z)+2f_{13}(z), \\
T(3)f_9(z)&=f_3(z)+7f_5(z)+9f_7(z)-12f_9(z)-4f_{11}(z)+2f_{13}(z), \\
T(3)f_{11}(z)&=5f_5(z)+5f_7(z)-7f_9(z)-2f_{11}(z)+f_{13}(z), \\
T(3)f_{13}(z)&=11f_5(z)+15f_7(z)-16f_9(z)-6f_{11}(z)+2f_{13}(z).
\end{align*}
Namely, $\hat{h}_i(z)\ (i=1,3,5,7,9,11,13)$ are eigen forms for 
$T(3)$.  Since the algebra of Hecke operators is 
commutative~\cite[Theorem~4.5.3]{Miyake}, 
$\hat{h}_i(z)\ (i=1,3,5,7,9,11,13)$ are normalized Hecke eigen forms. 
In addition, for each prime $p$ and $i=1,3,5,7,9,11,13$, 
\[
|c_{\hat{h}_i}(p)|\leq 2\sqrt{p}. 
\]
By (\ref{mat:Hecke}), we have 
\begin{align*}
\begin{pmatrix}
f_1(z)\\
f_3(z)\\
f_5(z)\\
f_7(z)\\
f_9(z)\\
f_{11}(z)\\
f_{13}(z)
\end{pmatrix}
=
\begin{pmatrix}
 0 & \frac{1}{3} & \frac{1}{2} & \frac{1}{4} & -\frac{1}{12} 
&-\frac{\sqrt{2}}{8}   & \frac{\sqrt{2}}{8}   \\
 -\frac{5}{24} & \frac{2}{3} & \frac{7}{8} & \frac{1}{3} & -\frac{2}{3} 
&\frac{-\sqrt{2}-4}{8} & \frac{\sqrt{2}-4}{8} \\
 -\frac{1}{20} & \frac{7}{15} & \frac{3}{4} & \frac{1}{4} & -\frac{5}{12} 
&\frac{-\sqrt{2}-4}{8} & \frac{\sqrt{2}-4}{8} \\
 -\frac{1}{24} & \frac{1}{3} & \frac{3}{8} & \frac{1}{6} & -\frac{1}{3} 
&\frac{-\sqrt{2}-2}{8} & \frac{\sqrt{2}-2}{8} \\
 -\frac{7}{60} & \frac{8}{15} & \frac{3}{4} & \frac{5}{12} & -\frac{7}{12} 
&\frac{-\sqrt{2}-4}{8} & \frac{\sqrt{2}-4}{8} \\
 -\frac{1}{10} & \frac{4}{15} & \frac{1}{2} & \frac{1}{4} & -\frac{5}{12} 
&\frac{-\sqrt{2}-2}{8} & \frac{\sqrt{2}-2}{8} \\
 -\frac{31}{120} & \frac{4}{5} & \frac{9}{8} & \frac{7}{12} & -\frac{3}{4} 
&\frac{-\sqrt{2}-3}{4} & \frac{\sqrt{2}-3}{4} \\
\end{pmatrix}
\begin{pmatrix}
\hat{h}_1(z)\\
\hat{h}_3(z)\\
\hat{h}_5(z)\\
\hat{h}_7(z)\\
\hat{h}_9(z)\\
\hat{h}_{11}(z)\\
\hat{h}_{13}(z)
\end{pmatrix}.
\end{align*}
Hence, we have
\begin{align*}
&{f_1}(z)+4{f_3}(z)-9{f_5}(z) 
+8{f_7}(z)-2{f_9}(z)-18{f_{11}}(z)+8{f_{13}}(z) 
\\
&=
-\frac{3}{4} \hat{h}_1(z)
+   2 \hat{h}_3(z)
-\frac{5}{4} \hat{h}_5(z)
+     \hat{h}_9(z).
\end{align*}
For each prime $p$, 
$|c_{f_1}(p)+4c_{f_3}(p)-9c_{f_5}(p)+8c_{f_7}(p)
-2c_{f_9}(p)-18c_{f_{11}}(p)+8c_{f_{13}}(p)|$ 
is bounded above by 
\[
\left(\frac{3}{4}+2+\frac{5}{4}+1 \right) 2\sqrt{p}
=10\sqrt{p}.
\]
Using (\ref{eqn:coef56}), $a_2(p)$ is bounded below by 
\begin{align}
\frac{4}{15}\left(p+1-10\sqrt{p}\right). \label{eqn:bound56}
\end{align}
Hence, (\ref{eqn:bound56}) is positive for $p >97$, 
namely, $a_2(p)>0$ for $p >97$. 
We have verified by {\sc Magma} that
$a_2(p)>0$
for each prime $p$ with $p \le 97$ and $p \ne 2,3,7,17,23$,
where $a_2(p)$ is listed in Table~\ref{Tab:an2} for
a prime $p \le 97$.
\end{proof}

\begin{table}[thb]
\caption{Coefficients $a_2(p)$ for primes $p \le 97$}
\label{Tab:an2}
\begin{center}
{\small
\begin{tabular}{cc|cc|cc|cc|cc}
\noalign{\hrule height0.8pt}
$p$ & $a_2(p)$ & $p$ & $a_2(p)$ &  $p$ & $a_2(p)$ & 
$p$ & $a_2(p)$ & $p$ & $a_2(p)$ \\
\hline
  2&  0& 13&  4& 31& 16& 53& 12& 73&  8\\
  3&  0& 17&  0& 37&  8& 59& 16& 79& 24\\
  5&  4& 19&  8& 41&  8& 61& 16& 83& 16\\
  7&  0& 23&  0& 43&  8& 67& 32& 89& 24\\
 11&  8& 29& 16& 47&  8& 71& 16& 97& 24\\
\noalign{\hrule height0.8pt}
\end{tabular}
}
\end{center}
\end{table}

\section{Lengths 32 and 40}\label{sec:3240}

In this section, we show the existence of an extremal
Type~II $\ZZ_{2k}$-code of lengths $32$ and $40$
for every positive integer $k$.
Our approach is similar to that in~\cite{Chapman, GH01}.

Some weighing matrix of order $16$
and weight $11$ is given in~\cite[p.~280]{Z4-H}
and it is denoted by $W_{16,11}$. 
Let $W_{20,11}$ be the $20 \times 20$ $(0,\pm1)$-matrix
$
\left(
\begin{array}{cc}
A & B \\
-B^T & A^T
\end{array}
\right),
$
where
$A$ and $B$ are negacirculant matrices with
first rows $r_A$ and $r_B$:
\[
r_A=(0,  1,  0,  1,  0,  1,  0,  1,  0,  1) \text{ and }
r_B=(1,  1, -1, -1,  1,  0,  0,  1,  0,  0),
\]
respectively.
Then $W_{n,11}$ ($n=16,20$) is a skew-symmetric
weighing matrix of order $n$ and weight $11$.
By Proposition~\ref{prop:constZ4}, $\tilde{C}_4(W_{n,11})$ 
($n=16,20$) is a Type~II $\ZZ_4$-code.
Moreover, 
$\tilde{C}_4(W_{16,11})$ is extremal~\cite{Z4-H}, and 
we have verified by {\sc Magma} that 
$\tilde{C}_4(W_{20,11})$ is extremal.
Hence, $A_4(\tilde{C}_4(W_{n,11}))$ is an extremal even unimodular
lattice in dimension $2n$ for $n=16,20$.

If $a,b,c$ and $d$ are integers with 
$c \equiv 2a+b \pmod 4$ and $d \equiv a+2b \pmod 4$,
then the extremal even unimodular lattice
$A_4(\tilde{C}_4(W_{n,11}))$ ($n=16,20$) contains a 
$\frac{1}{4}(a^2+11 b^2+c^2+11 d^2)$-frame by Proposition~\ref{prop:constZ4}.
Moreover, 
by Lemma~\ref{lem:frame} and Theorem~\ref{thm:prime11},
we have the following:

\begin{lem}\label{lem:3240-1}
$A_4(\tilde{C}_4(W_{n,11}))$
$(n=16,20)$
contains a $2k$-frame for every positive integer $k$
with $k \ge 2$ and $k \ne 11^m$, where $m$ is a positive integer.
\end{lem}
\begin{table}[thb]
\caption{Extremal Type~II $\ZZ_{22}$-codes of lengths $32,40$}
\label{Tab:3240}
\begin{center}
{\footnotesize
\begin{tabular}{c|l|l}
\noalign{\hrule height0.8pt}
Code & \multicolumn{1}{c|}{$r_A$} & \multicolumn{1}{c}{$r_B$} \\
\hline
$C_{22,32}$ &
$( 0,  0,  0,  0,  1, 10, 21,  6)$& $( 11,  2, 20, 20,  9,  3, 21, 11)$\\
$C_{22,40}$ &
$( 0,  0,  0,  0,  1, 19,  2, 10,  1,  6)$&
$( 13, 1, 10, 14, 10, 16, 13,  6,  9,  4)$ \\
\noalign{\hrule height0.8pt}
\end{tabular}
}
\end{center}
\end{table}

Let $C_{22,2n}$  ($n=16,20$) be the $\ZZ_{22}$-code of length $2n$
with generator matrix of the form (\ref{eq:GM}),
where the first rows $r_A$ and $r_B$ of negacirculant matrices
$A$ and $B$ are listed in Table~\ref{Tab:3240}.
These codes were found by considering pairs of binary 
Type~II codes and self-dual $\ZZ_{11}$-codes with 
generator matrices of the form (\ref{eq:GM})
(see~\cite[Theorem~2.3]{DHS} for the construction method).
Since $AA^T+BB^T=-I$ and the Euclidean weights of 
all rows of the generator matrix are divisible by $44$,
it follows from~\cite[Lemma~2.2]{BDHO} that $C_{22,2n}$ is 
a Type~II $\ZZ_{22}$-code of length $2n$ ($n=16,20$).
Moreover, we have verified by {\sc Magma} that $C_{22,2n}$ is extremal.
Hence, we have the following:

\begin{lem}\label{lem:3240-2}
$A_{22}(C_{22,2n})$ $(n=16,20)$ contains 
a $22k$-frame for every positive integer $k$.
\end{lem}

Since binary extremal Type~II codes of length $40$ are known,
by Lemmas~\ref{lem:2-1}, \ref{lem:3240-1} and \ref{lem:3240-2},
we have the following:

\begin{thm}\label{thm:3240}
There is an extremal Type~II $\ZZ_{2k}$-code of lengths $32,40$
for every positive integer $k$.
\end{thm}

\section{Length 48}\label{sec:48}
In this section and the next two sections, 
we show the existence of an extremal
Type~II $\ZZ_{2k}$-code of lengths $48,56$ and $64$
for every positive integer $k$.
Main tools of these sections are Proposition~\ref{prop:constZk}
and Theorems~\ref{thm:prime23} and~\ref{thm:prime29},
whereas the main tools in the previous section were
Proposition~\ref{prop:constZ4} and Theorem~\ref{thm:prime11}.

Let $W_{24,23}$ be the $24 \times 24$ $(0,\pm 1)$-matrix of the form:
\[
\left(
\begin{array}{cccc}
 0     & 1 & \cdots & 1 \\
-1     &   &        &   \\
\vdots &   &   A    &   \\
-1     &   &        &   
\end{array}
\right),
\]
where $A$ is the circulant matrix with first row
\[
(0, 1, 1, 1, 1,-1, 1,-1, 1, 1,-1,-1, 1, 1,-1,-1, 1,-1, 1,-1,-1,-1,-1). 
\]
Note that $1$'s are at the nonzero squares modulo $23$.
It is well known that $W_{24,23}$ is a skew-symmetric
weighing matrix of order $24$ and weight $23$.

If $a,b,c$ and $d$ are integers with 
$a \equiv d \pmod 3$ and
$b \equiv c \pmod 3$,
then the odd unimodular lattice  $A_3(C_3(W_{24,23}))$
contains a $\frac{1}{3}(a^2+23b^2+c^2+23d^2)$-frame
by Proposition~\ref{prop:constZk}.
Moreover, by Lemma~\ref{lem:frame} and 
Theorem~\ref{thm:prime23}, we have the following:

\begin{lem}
$A_3(C_3(W_{24,23}))$ contains a $k$-frame
for every positive integer $k$
with $k \ge 3$, 
$k\ne 2^{m_1}5^{m_2}7^{m_3}23^{m_4}$, 
where $m_i$ is a non-negative integer $(i=1,2,3,4)$.
\end{lem}

\begin{rem}
$C_3(W_{24,23})$ is a ternary extremal self-dual code of 
length $48$~\cite{Pless72} and it is called the Pless symmetry code.
\end{rem}

\begin{table}[thb]
\caption{Extremal Type~II $\ZZ_{2k}$-codes of length $48$ ($2k=14,46$)}
\label{Tab:48}
\begin{center}
{\footnotesize
\begin{tabular}{c|l|l}
\noalign{\hrule height0.8pt}
Code & \multicolumn{1}{c|}{$r_A$} & \multicolumn{1}{c}{$r_B$} \\
\hline
$C_{14,48}$ &
$(1,11, 6, 6, 9, 9,11, 1, 7, 0, 7, 7)$ &
$(5, 3,11,12, 4, 5, 1, 0, 4, 6,11, 8)$ \\
$C_{46,48}$ &
$( 0, 0, 1,26,29, 3,13,13,45,21, 0,23)$&
$(30, 9,23,37,33,37,35,40, 6, 8,33,28)$\\
\noalign{\hrule height0.8pt}
\end{tabular}
}
\end{center}
\end{table}

Let $C_{2k,48}$  ($2k=14,46$) be the $\ZZ_{2k}$-code of length $48$
with generator matrix of the form (\ref{eq:GM}),
where the first rows $r_A$ and $r_B$ of $A$ and $B$
are listed in Table~\ref{Tab:48}.
Similarly to Table~\ref{Tab:3240},
these codes were found by considering pairs of
binary Type~II codes and self-dual $\ZZ_{k}$-codes.
Since $AA^T+BB^T=-I$ and the Euclidean weights of 
all rows of the generator matrix are divisible by $4k$,
$C_{2k,48}$ is a Type~II $\ZZ_{2k}$-code of length $48$ and
$A_{2k}(C_{2k,48})$ is an even unimodular lattice $(2k=14,46)$.
In addition, we have verified by {\sc Magma} that 
$A_{2k}(C_{2k,48})$  $(2k=14,46)$ is extremal.
Hence, we have the following:

\begin{lem}\label{lem:48-1446}
For $2k=14,46$, $C_{2k,48}$ is 
an extremal Type~II $\ZZ_{2k}$-code of length $48$.
\end{lem}


The existence of an extremal Type~II $\ZZ_{2k}$-code of 
length $48$ is known for $k=1,2,\ldots,6$ (see~\cite[Table 1]{HM}).
Denote by $C_{2k,48}$ an existing extremal Type~II $\ZZ_{2k}$-code of 
length $48$ for $2k=8$ and $10$.
It is known that one of the two even unimodular neighbors 
$L_0 \cup L_1$ and $L_0 \cup L_3$, which are given in Section~\ref{sec:2U},
of $L=A_3(C_3(W_{24,23}))$ is extremal, and
this is denoted by $P_{48p}$ in~\cite{SPLAG} (see also~\cite{HKO}).
By Lemma~\ref{lem:neighbor},
we have the following:

\begin{lem}
\begin{itemize}
\item[\rm (1)]
$P_{48p}$ contains a $2k$-frame for every positive integer $k$
with $k \ge 3$, $k\ne 2^{m_1}5^{m_2}7^{m_3}23^{m_4}$, 
where $m_i$ is a non-negative integer $(i=1,2,3,4)$.
\item[\rm (2)]
$A_{8}(C_{8,48})$
contains an $8k$-frame for every positive integer $k$.
\item[\rm (3)]
$A_{10}(C_{10,48})$
contains a $10k$-frame for every positive integer $k$.
\item[\rm (4)]
$A_{14}(C_{14,48})$
contains a $14k$-frame for every positive integer $k$.
\item[\rm (5)]
$A_{46}(C_{46,48})$
contains a $46k$-frame for every positive integer $k$.
\end{itemize}
\end{lem}

Since $P_{48p}$ and $A_{2k}(C_{2k,48})$  ($2k=8,10,14,46$)
are extremal even unimodular lattices,
by Lemma~\ref{lem:2-1}
we have the following:

\begin{thm}\label{thm:48}
There is an extremal Type~II $\ZZ_{2k}$-code of
length $48$ for every positive integer $k$.
\end{thm}

\begin{rem}
Three non-isomorphic extremal even unimodular lattices
are known for dimension $48$ (see~\cite{SPLAG}).
It is worthwhile to determine whether 
$A_{14}(C_{14,48})$ and $A_{46}(C_{46,48})$ are new or not.
\end{rem}

\section{Length 56}\label{sec:56}

Let $D_{28}$ be the $28 \times 28$ 
$(0,\pm1,\pm2)$-matrix 
$
\left(
\begin{array}{cc}
A & B \\
-B^T & A^T
\end{array}
\right),
$
where
$A$ and $B$ are negacirculant matrices with
first rows $r_A$ and $r_B$:
\begin{multline*}
r_A=( 0, 1, 2, 0, 0, 0, 0, 0, 0, 0, 0, 0, 2, 1) \text{ and }\\
r_B=(-1, 0, 1,-1, 0, 2,-1, 2, 0, 1, 1, 0,-1, 2),
\end{multline*}
respectively.
The matrix $D_{28}$ satisfies that
$D_{28} D_{28}^T=29 I$ and
$D_{28}^T=-D_{28}$.

If $a,b,c$ and $d$ are integers with 
$a \equiv d \pmod 5$ and
$b \equiv c \pmod 5$, then
the odd unimodular lattice 
$A_5(C_5(D_{28}))$ contains a 
$\frac{1}{5}(a^2+29b^2+c^2+29d^2)$-frame
by Proposition~\ref{prop:constZk}.
Moreover, by Lemma~\ref{lem:frame} and Theorem~\ref{thm:prime29}, 
we have the following:

\begin{lem}
$A_5(C_5(D_{28}))$ contains a $k$-frame
for every positive integer $k$ with $k \ge 5$, 
$k\ne 2^{m_1}3^{m_2}7^{m_3}17^{m_4}23^{m_5}$, 
where $m_i$ is a non-negative integer $(i=1,2,3,4,5)$.
\end{lem}

\begin{rem}
The weight enumerator of a code
is given by $\sum_{i} A_i y^i$, where $A_i$ denotes the number
of codewords of weight $i$.
We have verified by {\sc Magma} that
$C_5(D_{28})$ has weight enumerator 
$1+ 168 y^{12} + 224 y^{14} + 448 y^{15} +9464  y^{16}
+ \cdots$ and $A_5(C_5(D_{28}))$ has minimum norm $5$.
\end{rem}

\begin{rem}
We have verified by {\sc Magma} that every ternary self-dual
code $C_3(W)$ has minimum weight less than $15$
if $W$ is a skew-symmetric weighing matrix
of order $28$ and weight $k \equiv 2 \pmod 3$,
which has the form
$
\left(
\begin{array}{cc}
A & B \\
-B^T & A^T
\end{array}
\right),
$
where $A$ and $B$ are negacirculant matrices.
This is a reason to find the above matrix $D_{28}$.
\end{rem}

\begin{table}[thb]
\caption{Extremal Type~II $\ZZ_{2k}$-codes of length $56$ ($2k=14,34,46$)}
\label{Tab:56}
\begin{center}
{\scriptsize
\begin{tabular}{c|l|l}
\noalign{\hrule height0.8pt}
$2k$& \multicolumn{1}{c|}{$r_A$} & \multicolumn{1}{c}{$r_B$} \\
\hline
14 & $(0,0,0,0,1,9,0,12,8,10,12,9,12,7)$& 
$(8,1,13,4,9,7,9,0,10,7,0,0,8,1)$ \\
34 & $(0,0,0,0,1,13,6,1,4,0,6,4,22,8)$
& $(17,32,4,30,1,1,6,32,31,23,23,14,9,27)$ \\
46 & $(0,0,0,0,1,23,6,33,43,19,30,18,29,11)$ &
$(17,13,32,23,42,16,38,31,29,1,30,25,41,22)$ \\
\noalign{\hrule height0.8pt}
\end{tabular}
}
\end{center}
\end{table}
Let $C_{2k,56}$  ($2k=14,34,46$) be the $\ZZ_{2k}$-code of length $56$
with generator matrix of the form (\ref{eq:GM}),
where the first rows $r_A$ and $r_B$ of $A$ and $B$
are listed in Table~\ref{Tab:56}.
These codes were found by considering pairs of
binary Type~II codes and self-dual $\ZZ_{k}$-codes.
Since $AA^T+BB^T=-I$ and the Euclidean weights of 
all rows of the generator matrix are divisible by $4k$,
$C_{2k,56}$ is a Type~II $\ZZ_{2k}$-code of length $56$ and 
$A_{2k}(C_{2k,56})$ is an even unimodular lattice  ($2k=14,34,46$).
Moreover, we have verified by {\sc Magma} that $A_{2k}(C_{2k,56})$ 
is extremal.  Hence, we have the following:

\begin{lem}\label{lem:56-14}
For $2k=14,34,46$,
$C_{2k,56}$ is an extremal Type~II $\ZZ_{2k}$-code of length $56$.
\end{lem}

We have verified by {\sc Magma} that one of the
even unimodular neighbors $L_0 \cup L_1$ and $L_0 \cup L_3$
of $L=A_5(C_5(D_{28}))$ is extremal.
Denote by $L_{56}$ 
the extremal even unimodular neighbor of $A_5(C_5(D_{28}))$.
The existence of an extremal Type~II $\ZZ_{2k}$-code of 
length $56$ is known for $k=1,2,\ldots,6$ (see~\cite[Table 1]{HM}).
Denote by $C_{2k,56}$ an existing extremal Type~II $\ZZ_{2k}$-code of 
length $56$ for $2k=6$ and $8$.
By Lemma~\ref{lem:neighbor}, we have the following:

\begin{lem}
\begin{itemize}
\item[\rm (1)]
$L_{56}$ contains a $2k$-frame for every positive integer $k$
with $k \ge 5$, 
$k\ne 2^{m_1}3^{m_2}7^{m_3}17^{m_4}23^{m_5}$, 
where $m_i$ is a non-negative integer $(i=1,2,3,4,5)$.
\item[\rm (2)]
$A_6(C_{6,56})$ contains a $6k$-frame for every positive integer $k$.
\item[\rm (3)]
$A_8(C_{8,56})$ contains an $8k$-frame for every positive integer $k$.
\item[\rm (4)]
$A_{14}(C_{14,56})$ contains a $14k$-frame for every positive integer $k$.
\item[\rm (5)]
$A_{34}(C_{34,56})$ contains a $34k$-frame for every positive integer $k$.
\item[\rm (6)]
$A_{46}(C_{46,56})$ contains a $46k$-frame for every positive integer $k$.
\end{itemize}
\end{lem}

Since $L_{56}$ and $A_{2k}(C_{2k,56})$ $(2k=6,8,14,34,46)$ are
extremal even unimodular lattices,
by Lemma~\ref{lem:2-1} we have the following:

\begin{thm}\label{thm:56}
There is an extremal Type~II $\ZZ_{2k}$-code of
length $56$ for every positive integer $k$.
\end{thm}

\section{Length 64}\label{sec:64}

Let $W_{32,23}$ and  $W_{32,17}$ 
be the $32 \times 32$ $(0,\pm1)$-matrices
$
\left(
\begin{array}{cc}
A & B \\
-B^T & A^T
\end{array}
\right),
$
where
$A,B$ are negacirculant matrices with
the following first rows $r_A,r_B$:
\begin{multline*}
r_A= (0,1,1,0,-1,1,-1,0,0,0,-1,1,-1,0,1,1), \\
r_B= (0,1,0,1,1,1,1,-1,0,-1,-1,1,-1,-1,-1,1)
\end{multline*}
and
\begin{multline*}
r_A= (0,0,1,1,0,0,0,0,1,0,0,0,0,1,1,0),\\
r_B= (0,1,-1,-1,0,-1,1,0,1,-1,1,1,1,0,-1,1),
\end{multline*}
respectively.
The matrix $W_{32,23}$ (resp.\ $W_{32,17}$) is a skew-symmetric
weighing matrix of order $32$ and weight $23$ (resp.\ $17$).

If $a,b,c$ and $d$ are integers with 
$a \equiv d \pmod 3$ and
$b \equiv c \pmod 3$, then 
the odd unimodular lattice  $A_3(C_3(W_{32,23}))$ contains
a $\frac{1}{3}(a^2+23b^2+c^2+23d^2)$-frame
by Proposition~\ref{prop:constZk}.
Moreover, 
by Lemma~\ref{lem:frame} and Theorem~\ref{thm:prime23}, we have the following:

\begin{lem}
$A_3(C_3(W_{32,23}))$ contains a $k$-frame for every positive integer $k$
with $k \ge 3$, $k\ne 2^{m_1}5^{m_2}7^{m_3}23^{m_4}$, 
where $m_i$ is a non-negative integer $(i=1,2,3,4)$.
\end{lem}

\begin{lem}
$A_3(C_3(W_{32,17}))$ contains a 
$7k$-frame and a $23k$-frame for every positive integer $k$.
\end{lem}
\begin{proof}
Take $(a,b,c,d)=(0, 1, -2, 0)$ and
$(0, 2, -1, 0)$.  
By Proposition~\ref{prop:constZk},
the odd unimodular lattice $A_3(C_3(W_{32,17}))$ contains a 
$7$-frame and a $23$-frame.
The result follows by Lemma~\ref{lem:frame}.
\end{proof}


Denote by $L_{64,t}$ any of two even unimodular neighbors 
$L_0 \cup L_1$ and $L_0 \cup L_3$ of $L=A_3(C_3(W_{32,t}))$ $(t=23,17)$.
We have verified by {\sc Magma} that
$C_3(W_{32,t})$ ($t=23, 17$) have minimum weight $15$.
Hence, $L_{64,t}$ are extremal (see~\cite[Theorem~6]{HKO}).
The existence of an extremal Type~II $\ZZ_{2k}$-code of 
length $64$ is known for $k=1,2,\ldots,6$ (see~\cite[Table 1]{HM}).
Denote by $C_{2k,64}$ an existing extremal Type~II $\ZZ_{2k}$-code of 
length $64$ for $2k=8$ and $10$.
By Lemma~\ref{lem:neighbor}, we have the following:

\begin{lem}
\begin{itemize}
\item[\rm (1)]
$L_{64,23}$
contains a $2k$-frame for every positive integer $k$
with $k \ge 3$, $k\ne 2^{m_1}5^{m_2}7^{m_3}23^{m_4}$, 
where $m_i$ is a non-negative integer $(i=1,2,3,4)$.
\item[\rm (2)]
$L_{64,17}$ contains a $14k$-frame and 
a $46k$-frame for every positive integer $k$.
\item[\rm (3)]
$A_{8}(C_{8,64})$  contains an $8k$-frame
for every positive integer $k$.
\item[\rm (4)]
$A_{10}(C_{10,64})$  contains a $10k$-frame
for every positive integer $k$.
\end{itemize}
\end{lem}
\begin{rem}
We have found an extremal Type~II $\ZZ_{2k}$-code $C_{2k,64}$ ($2k=14,46$)
of length $64$
explicitly.
These codes have generator matrices of the form (\ref{eq:GM}),
where the first rows $r_A$ and $r_B$ of negacirculant matrices
$A$ and $B$ are as follows:
\begin{multline*}
r_A=(0,0,0,0,1,9,1,2,7,9,13,0,10,3,10,1)
\text{ and } \\
r_B=(0,5,12,13,5,6,8,8,1,10,8,1,3,0,8,3)
\end{multline*}
and
\begin{multline*}
r_A=(0,0,0,0,1,10,24,27,35,22,7,20,22,7,36,18) \text{ and }
\\
r_B=(5,15,19,43,19,18,35,5,5,42,34,27,23,36,4,32)
\end{multline*}
respectively.
$A_{2k}(C_{2k,64})$ ($2k=14,46$) contains 
a $2km$-frame for every positive integer $m$.
\end{rem}

Since $L_{64,t}$ $(t=17,23)$ and 
$A_{2k}(C_{2k,64})$  ($2k=8,10$) are extremal even unimodular lattices,
by Lemma~\ref{lem:2-1} we have the following:

\begin{thm}\label{thm:64}
There is an extremal Type~II $\ZZ_{2k}$-code of
length $64$ for every positive integer $k$.
\end{thm}

\section{Length 72}\label{sec:72}
Although the approach is somewhat different to those in previous sections, 
it is shown in this section that
there is an extremal Type~II $\ZZ_{4k}$-code of length
$72$ for every positive integer $k$ with $k \ge 2$.

The existence of an extremal even unimodular lattice
in dimension $72$ was a long-standing open question.
Recently, the first example of such a lattice has been
found by Nebe~\cite{Nebe72}.
Roughly speaking, using the Leech lattice  $\Lambda_{24}$, 
Nebe~\cite{Nebe72} found a pair $(M,N)$ such that
\[
\frac{1}{\sqrt{2}}
\{(x_1+y,x_2+y,x_3+y) \mid 
x_1,x_2,x_3 \in M, y \in N, x_1+x_2+x_3 \in M \cap N\}
\]
is an  extremal even unimodular lattice
in dimension $72$,
where
$M \cong N \cong \sqrt{2} \Lambda_{24}$,
$\Lambda_{24}=M+N$ and $2\Lambda_{24}=M \cap N$.
Hence, the extremal even unimodular lattice $N_{72}$
in dimension $72$, which has been found by Nebe~\cite{Nebe72},
contains a sublattice 
\[
\frac{1}{\sqrt{2}}
\{(x,\0,\0),(\0,y,\0),(\0,\0,z) \mid x,y,z \in 2\Lambda_{24}\},
\]
where $\0$ denotes the zero vector of length $24$.
As described in Section~\ref{Sec:1}, it was shown 
in~\cite{Chapman,GH01} that $\Lambda_{24}$ contains a $2k$-frame
for every positive integer $k$ with $k \ge 2$.
Let $\{f_{1},\ldots,f_{24}\}$ be a $2k$-frame of $\Lambda_{24}$
($k \ge 2$).
Then
\[
\{(\sqrt{2}f_{i},\0,\0),(\0,\sqrt{2}f_{i},\0),(\0,\0,\sqrt{2}f_{i}) \mid
i=1,2,\ldots,24\}
\]
is a $4k$-frame of $N_{72}$ for $k \ge 2$.
Therefore, we have the following:

\begin{thm}\label{thm:72}
There is an extremal Type~II $\ZZ_{4k}$-code of length
$72$ for every positive integer $k$ with $k \ge 2$.
\end{thm}

\begin{rem}
It is worthwhile to determine whether there is an 
extremal Type~II $\ZZ_{2m}$-code of length $72$ or not
for the cases that $m=2$ and $m$ is an odd positive integer.
Of course, the case $m=1$ is a long-standing open 
question (see~\cite{Huffman05, RS-Handbook}).
\end{rem}

As an example, an explicit generator matrix
$\left(\begin{array}{cc}
I & M_{8,72}
\end{array}\right)$ of some
extremal Type~II $\ZZ_{8}$-code $C_{8,72}$ of length $72$
with $A_8(C_{8,72}) \cong N_{72}$
is given by listing $M_{8,72}$ in Figure~\ref{Fig:72}.

\begin{figure}[thbp]
\centering
{\footnotesize
\[
\left(
\begin{array}{c}
000404444444276260502007044004400004\\
444400400444035522611141064440040444\\
004040004004725430437502014040004004\\
044444404444437162115704404040440400\\
444440004044402144472442074400444044\\
004404444040142273042604020004400400\\
004400044040064160263204410444404444\\
044404004044707461671100414000040000\\
400404044000207033255140074004000440\\
040040004044445560422004010000444044\\
000000044040266504204702414044404444\\
444444404440401366111547444040440404\\
400404444440404004044424042414230661\\
400044004000400400404404604064336223\\
444440040004444444044424744660227162\\
400000000004400000400000404030136026\\
040444040440040404440054340031762541\\
440404444404444004440454347236540305\\
044040404400044440404040705211520543\\
000000044040000004044060304300622360\\
444000040044440044404050205221242642\\
404404400404404040400444046756741311\\
444444440400440444444020544732763677\\
004400000000000040000020142220444752\\
272420542247004004444404044004400004\\
430216002627000400004404400400400040\\
723141607130004440004404004040004004\\
034454706026400000400040000000000004\\
002357524571004000400404400440004440\\
143224276042004044044004040004400400\\
464551763307440404444040004404044040\\
706143413552044040444000444000040000\\
203707035152404004040040044004000440\\
044516363535404040400444444040004440\\
666517121453440444444040040004044040\\
402433523733440440400040444040440404
\end{array} \right)
\]
\caption{A generator matrix of $C_{8,72}$}
\label{Fig:72}
}
\end{figure}


\section{Positivity of coefficients of theta series}\label{sec:T}
In this section, we discuss the
positivity of coefficients of theta series of extremal even unimodular
lattices.
As we have already mentioned in Section~\ref{sec:2U}, 
it is important to study
the positivity and non-negativity of coefficients of 
the theta series of extremal even unimodular lattices.

The positivity is also useful to construct 
spherical $t$-designs
(see~\cite{BB09} for a recent survey on this subject).
For a lattice $L$ and a positive integer $m$, 
the {\em shell} of norm $m$ of $L$ is defined by 
$\{x\in L \mid (x,x)=m \}$. 
Then shells of a lattice often give examples of 
spherical $t$-designs for some $t$. 
For example, any nonempty shell 
of an extremal even unimodular lattice in dimension $n$ 
forms a spherical $t$-design, where
$t=11$ if $n \equiv 0 \pmod{24}$,
$t=7$ if $n \equiv 8 \pmod{24}$,
$t=3$ if $n \equiv 16 \pmod{24}$~\cite{Venkov} 
(see also~\cite[Theorem~3.8]{BB09}). 
Note that $11$ is the largest $t$ among known 
spherical $t$-designs constructed as the shells of some lattices.
The positivity of some coefficients of the theta series of 
extremal even unimodular lattices means that 
the corresponding shells of those lattices are not empty sets. 

Let $A_{2m}$ denote the number of vectors of norm $2m$ in 
an extremal even unimodular lattice $L$
in dimensions $n$, where $n=8,16$ and $24$.
For the cases $n=8$ and $16$, it follows from (\ref{eq:T2}) that
$A_{2m} > 0$ for $m \ge 0$.
For the case that $n=24$, that is, $L$ is the Leech lattice,
it is well known that $A_{2m} > 0$ for $m \ge 2$ (see~\cite[p.~51]{SPLAG}).

As a consequence of 
Theorems~\ref{thm:3240}, \ref{thm:48}, \ref{thm:56}, \ref{thm:64}
and \ref{thm:72}, and the
Fisher type bound (see~\cite[Theorem~2.12]{BB09}), 
we give the following observation on the
positivity of coefficients of 
the theta series of extremal even unimodular lattices.

\begin{cor}\label{cor:posi}
Let $L$ be an extremal even unimodular lattice in dimension $n$.
Let $A_{2m}$ denote the number of vectors of norm $2m$ in $L$,
that is, $|L_{2m}|=A_{2m}$.
Then
\[
A_{2m} \ge 
\left\{
\begin{array}{lcl}
  11968 &\text{ if } &n=32\\
     80 &\text{ if } &n=40\\
5197920 &\text{ if } &n=48\\
  61712 &\text{ if } &n=56\\
    128 &\text{ if } &n=64
\end{array}
\right.
\]
for every positive integer $m$ with $m \ge \lfloor \frac{n}{24} \rfloor+1$.
If $n=72$, then
\[
A_{4m} \ge 36949680
\]
for every positive integer $m$ with $m \ge 2$.
\end{cor}
\begin{proof}
Suppose that $n=32,40,48,56, 64$.
By Theorems~\ref{thm:3240}, \ref{thm:48}, \ref{thm:56} and \ref{thm:64},
$L_{2m}$ is a nonempty shell 
for every positive integer $m$
with $m \ge \lfloor \frac{n}{24} \rfloor+1$.
Hence, $L_{2m}$ is a spherical $t$-design,
where $t=7,3,11,7,3$, respectively.
If $L_{2m}$ is a spherical $(2e+1)$-design
then $|L_{2m}| \ge 2 \binom{n+e-1}{e}$
(see~\cite[Theorem~2.12]{BB09}).
The case $n=72$ is similar.
The result follows.
\end{proof}


\bigskip
\noindent
{\bf Acknowledgments.}
The authors would like to thank Naoki Murabayashi,
Manabu Oura for helpful discussions, and 
Gabriele Nebe
for useful comments and providing a preprint of~\cite{JR}. 
The authors would also like to thank the anonymous referee
for beneficial comments on an earlier version of the manuscript.
This work was supported by JST PRESTO program and
JSPS KAKENHI Grant Number 22840003, 23340021, 23654029, 24740031.


\end{document}